\documentclass{amsart}
\pdfoutput=1
\usepackage[utf8]{inputenc}
\usepackage[T1]{fontenc}
\usepackage[english]{babel}
\usepackage[babel=true]{microtype}
\usepackage{amsmath,amsthm,amssymb,bookmark,float,tikz}
\usepackage[backend=biber,style=alphabetic,giveninits]{biblatex}
\DeclareNameAlias{default}{family-given}
\theoremstyle{plain}
\newtheorem{theorem}{Theorem}[section]
\newtheorem{lemma}[theorem]{Lemma}
\newtheorem{corollary}[theorem]{Corollary}
\newtheorem{conjecture}[theorem]{Conjecture}
\newtheorem{proposition}[theorem]{Proposition}
\newtheorem{principle}[theorem]{Principle}

\theoremstyle{definition}
\newtheorem{definition}[theorem]{Definition}
\newtheorem{example}[theorem]{Example}
\newtheorem{notation}[theorem]{Notation}
\newtheorem{problem}[theorem]{Problem}

\theoremstyle{remark}
\newtheorem{remark}[theorem]{Remark}

\newtheorem*{acknowledgements}{Acknowledgements}
\newtheorem*{structure}{Structure of the paper}

\title[On the existence of nef-partitions for smooth well-formed Fano WCI]{On the existence of nef-partitions for smooth \\ well-formed Fano weighted complete intersections}
\author{Mikhail Ovcharenko}
\address{Steklov Mathematical Institute of Russian Academy of Sciences, 8 Gubkina street, Moscow, Russia}
\email{ovcharenko@mi-ras.ru}

\DeclareMathOperator{\Bs}{Bs}
\DeclareMathOperator{\codim}{codim}
\DeclareMathOperator*{\lcm}{lcm}
\DeclareMathOperator{\Facets}{Facets}
\DeclareMathOperator{\im}{im}

\DeclareMathOperator{\Proj}{Proj}
\DeclareMathOperator{\red}{red}
\DeclareMathOperator{\Spec}{Spec}
\DeclareMathOperator{\Sing}{Sing}
\DeclareMathOperator{\Vertices}{Vert}
\DeclareMathOperator{\vertices}{vert}

\calclayout

\begin{document}

\begin{abstract}
  A nef-partition for a weighted complete intersection is a combinatorial structure on its weights and degrees which is important for Mirror Symmetry. It is known that nef-partitions exist for smooth well-formed Fano weighted complete intersections of small dimension or codimension, and that in these cases they are \emph{strong} in the sense that they can be realized as fibers of morphisms of \emph{weighted simplicial complexes}, i.e., finite abstract simplicial complexes equipped with a weight function.

  It was conjectured that this approach can be extended to the case of arbitrary codimension. We show that in the case of any codimension greater than 3 strong nef-partitions may not exist, and provide a sufficient combinatorial condition for existence of a strong nef-partition in terms of weights.

  We also show that the combinatorics of smooth well-formed weighted complete intersections can be arbitrarily complicated from the point of view of simplicial geometry.
\end{abstract}

\maketitle

\section{Introduction}

Complete intersections in weighted projective spaces, also known as \emph{weighted complete intersections}, form a large class of varieties allowing one to carry out various explicit computations (see Section~\ref{section:preliminaries} for a brief review). In this paper we consider the combinatorics of smooth \emph{well-formed} weighted complete intersections (see Definition~\ref{definition:WF-subscheme}). We are especially interested in \emph{Fano varieties}, i.e., in complete varieties whose anticanonical class is ample. The property of a weighted complete intersection to be smooth and well-formed imposes strong restrictions on a possible choice of weights and degrees. For example, it is possible to explicitly classify smooth well-formed Fano weighted complete intersections (see~\cites{przyjalkowski/bounds,przyjalkowski/codimension,ovcharenko/classification}).

In the paper we study the existence of so-called \emph{nef-partitions}, which are important for Mirror Symmetry. Everything is over the field \(\mathbb{C}\). Our approach is mainly combinatorial, hence most of the statements hold in much more general case. Throughout the paper \(a^{(t)}\) stands for \(a \in \mathbb{Z}_{> 0}\) repeated \(t\) times.

\begin{definition}[see~{\cite[Definition~1.1]{przyjalkowski/nef}},~{\cite{givental/mirror}},~{\cite{batyrev/degenerations}}]
  Let \(X \subset \mathbb{P}(a_0, \ldots, a_N)\) be a smooth well-formed Fano weighted complete intersection of multidegree \((d_1, \ldots, d_c)\). A \emph{nef-partition} for \(X\) is a splitting
  \[
    \{0, \ldots, N\} = I_0 \sqcup \ldots \sqcup I_c, \quad \sum_{i \in I_j} a_i = d_j, \quad j = 1, \ldots, c.
  \]
  A nef-partition is called \emph{nice} if there exists an index \(i \in I_0\) such that \(a_i = 1\).
\end{definition}

Recall that Mirror Symmetry corresponds to a Fano variety a certain one-dimensional family of Calabi--Yau varieties called its \emph{Landau--Ginzburg model}. Elements of such family should be mirror dual to anticanonical sections of the Fano variety. This notion clearly depends on the definition of a ``mirror dual'' object. The existence of a nef-partition allows one to construct a Landau--Ginzburg model for a smooth well-formed Fano weighted complete intersection in the sense of Givental (see~\cite{givental/mirror} and~\cite[Definition~2.1]{przyjalkowski/nef}).

If a nef-partition is nice, then one can birationally represent Givental's Landau--Ginzburg model by a complex torus equipped with a regular function on it (see~\cite[Theorem~9]{przyjalkowski/hori}). Moreover, the Newton polytope of the corresponding Laurent polynomial is always realized as the fan polytope of a toric degeneration of the smooth well-formed Fano weighted complete intersection (see~\cite[Theorem~2.2]{ilten/degenerations}).

All this motivates the following combinatorial conjecture.

\begin{conjecture}[see~{\cite[Conjecture~1.5]{przyjalkowski/nef}}]\label{conjecture:main}
  A smooth well-formed Fano weighted complete intersection admits a nice nef-partition.
\end{conjecture}

This conjecture is known to be true for Fano weighted complete intersections of small dimension or codimension (see~\cites{przyjalkowski/bounds,przyjalkowski/nef}). In these cases a nef-partition satisfies the following additional property.

\begin{definition}[see~{\cite[Remark~4.2]{przyjalkowski/singular}}]
  Let \(X \subset \mathbb{P}(a_0, \ldots, a_N)\) be a smooth well-formed Fano weighted complete intersection of multidegree \((d_1, \ldots, d_c)\). A nef-partition \(I_0 \sqcup \ldots \sqcup I_c\) is called \emph{strong} if
  \[
    a_i = 1 \text{ for any } i \in I_0; \quad
    a_i \mid d_j \text{ for any } j = 1, \ldots, c \text{ and } i \in I_j.
  \]
\end{definition}

\begin{theorem}[see~{\cite[Theorems~1,2]{przyjalkowski/nef}}]\label{theorem:nef}
  Let \(X \subset \mathbb{P}(a_0, \ldots, a_N)\) be a smooth well-formed Fano weighted complete intersection of codimension 1 or 2. Then there exists a strong nef-partition for \(X\).
\end{theorem}

Strong nef-partitions are not only convenient from the combinatorial point of view, but they also appear naturally in Przyjalkowski's construction of \emph{log Calabi--Yau compactifications} of a Landau--Ginzburg model for a smooth well-formed Fano weighted complete intersection (see~\cite[Remark~4.2]{przyjalkowski/singular}).

The main observation of the proof of Theorem~\ref{theorem:nef} is that the singular locus \(\Sing(\mathbb{P}(\rho))\) of a weighted projective space with weights \(\rho = (a_0, \ldots, a_N)\) can be interpreted as a \emph{weighted simplicial complex} \(\mathcal{S}(\rho)\), i.e., an abstract simplicial complex equipped with a weight function (we refer the reader to Section~\ref{section:WSC} for definitions and details). Moreover, base loci of linear systems \(\vert \mathcal{O}_{\mathbb{P}(\rho)}(d) \vert\) can also be interpreted in this way (see Lemma~\ref{lemma:face-ring-proj} for a precise statement). This leads us to the following problem.

\begin{problem}[see~{\cite[Section~4]{przyjalkowski/nef}}]\label{problem:nef}
  Could one construct a nef-partition for a smooth well-formed Fano weighted complete intersection of multidegree \(\mu\) as fibers of a weighted simplicial map \(\mathcal{S}(\rho) \rightarrow \mathcal{S}(\mu)\)?
\end{problem}

\begin{remark}
  Note that any such nef-partition would be strong (see Definition~\ref{definition:WSC} and Notation~\ref{notation:S-complex}). 
\end{remark}

The answer to Problem~\ref{problem:nef} is negative: in the case of any codimension greater than 3 there exist smooth well-formed Fano weighted complete intersections without a strong nef-partition.

\begin{example}\label{example:no-strong-nef-partition}
  Put \(\rho = (1^{(61 + t)}, 6, 10, 15)\) for any \(t > 0\). Let \(X \subset \mathbb{P}(\rho)\)  be a general weighted complete intersection of multidegree \((16, 21, 25, 30)\) (see~\cite[Lemma~2.18]{ovcharenko/classification}). By Proposition~\ref{proposition:WCI-smooth-WF-criterion} the variety \(X\) is smooth and well-formed if and only if it is \emph{quasi-smooth} (see Definition~\ref{definition:WPV-quasismooth}), and \(X \cap \Sing(\mathbb{P}(\rho)) = \varnothing\). The former condition holds by~\cite[Proposition~3.1]{pizzato/nonvanishing}, and the latter one can be checked directly using Lemma~\ref{lemma:WPS-WF-singular-locus}. Moreover, \(X\) is Fano of Fano index \(t\) for any \(t > 0\) (see~{\cite[Theorem~3.3.4]{dolgachev/weighted}},~{\cite[\nopp 6.14]{ianofletcher/weighted}}). Nonetheless, \(X\) does not have a strong nef-partition.

  This example can be generalized to arbitrary codimension as a series of weighted complete intersections \(X \subset \mathbb{P}(1^{(61 + t + 2m)}, 6, 10, 15)\) of multidegree \((2^{(m)}, 16, 21, 25, 30)\) (see~\cite{ovcharenko/classification} for further details).
\end{example}

Our main result is a sufficient condition for existence of a strong nef-partition only in terms of weights, without any assumptions on multidegree of a weighted complete intersection.

\begin{definition}\label{definition:divisible-subset}
  For any tuple \(\rho = (a_0, \ldots, a_N) \in \mathbb{Z}_{> 0}^{N + 1}\) we refer to a subset \(I \subset \{0, \ldots, N\}\) as
  \begin{itemize}
  \item \emph{non-divisible} if we have \(a_i \nmid a_j\) for any \(i, j \in I\) with \(i \neq j\);
  \item \emph{strongly non-divisible} if \(a_k\) does not divide \(\lcm_{i, j \in I} (\gcd_{i \neq j}(a_i, a_j))\) for any \(k \in I\).
  \end{itemize}
\end{definition}

Note that any strongly non-divisible subset is non-divisible.

\begin{example}\label{example:non-divisible}
  On the one hand, let \((q_1, \ldots, q_l) \in \mathbb{Z}_{> 1}^l\) be pairwise coprime integers, and \((m_1, \ldots, m_l) \in \mathbb{Z}_{> 0}^l\) be any tuple. Put \(\rho = (q_1^{(m_1)}, \ldots, q_l^{(m_l)})\). Then any non-divisible subset for \(\rho\) is strongly non-divisible.

  On the other hand, put \(\rho = (1^{(t + 1)}, 6, 10, 15)\) for any \(t \geqslant 0\). Then the subset \(\{t + 1, t + 2, t + 3\}\) is non-divisible, but not strongly non-divisible.
\end{example}

If the whole weight tuple \(\rho\) is strongly non-divisible, it is not hard to derive the existence of a strong nef-partition from~\cite[Lemma~6.1]{pizzato/nonvanishing} and Proposition~\ref{proposition:WCI-smooth-WF-combinatorics}. If there exists a non-divisible subset for the tuple \(\rho\) which is not strongly non-divisible, there could exist smooth well-formed Fano weighted complete intersections without a strong nef-partition (see Examples~\ref{example:no-strong-nef-partition} and~\ref{example:non-divisible}). We claim that if any non-divisible subset for \(\rho\) is strongly non-divisible, then a strong nef-partition exists for any possible multidegree.

\begin{theorem}[see the proof in Section~\ref{section:proof}]\label{theorem:main}
  Let \(X \subset \mathbb{P}(a_0, \ldots, a_N)\) be a smooth well-formed Fano weighted complete intersection which is not an intersection with a linear cone. Assume that any non-divisible subset \(I \subset \{0, \ldots, N\}\) is strongly non-divisible. Then \(X\) admits an explicitly constructed strong nef-partition.
\end{theorem}

\begin{remark}
  In particular, Theorem~\ref{theorem:main} implies that under a certain restrictions on weights the property of a weighted complete intersection to have a strong nef-partition does not depend on multidegree at all, only on its property to be smooth, well-formed, and Fano, which is unexpected.
\end{remark}

\begin{remark}
  For any tuple \(\rho \in \mathbb{Z}_{> 0}^{N + 1}\) the set \(\mathcal{A}(\rho)\) of all non-divisible subsets and the set \(\mathcal{B}(\rho)\) of all strongly non-divisible subsets are finite simplicial complexes on \(\{0, \ldots, N\}\), one being a subcomplex of another. In other words, \((\mathcal{A}(\rho), \mathcal{B}(\rho))\) is a simplicial pair. In Theorem~\ref{theorem:main} we require precisely that \(\mathcal{A}(\rho) = \mathcal{B}(\rho)\), i.e., that the simplicial pair \((\mathcal{A}(\rho), \mathcal{B}(\rho))\) is trivial.
\end{remark}

We show that the combinatorics of a smooth well-formed weighted complete intersection \(X \subset \mathbb{P}(\rho)\) of multidegree \(\mu\) can be arbitrarily complicated from the point of view of simplicial geometry. More precisely, a weighted simplicial map \((\mathcal{S}(\rho), \Phi_{\mathcal{S}(\rho)}) \rightarrow (\mathcal{S}(\mu), \Phi_{\mathcal{S}(\mu)})\) can show the following pathological behavior.

\begin{example}\label{example:ill-defined}
  Let \(X_{\mu} \subset \mathbb{P}(\rho)\) be a smooth well-formed weighted complete intersection of multidegree \(\mu\).
  \begin{enumerate}
  \item There exists a series of examples of smooth well-formed weighted complete intersections such that any weighted simplicial map \((\mathcal{S}(\rho), \Phi_{\mathcal{S}(\rho)}) \rightarrow (\mathcal{S}(\mu), \Phi_{\mathcal{S}(\mu)})\) is a map to a point (see Example~\ref{example:contraction}).
  \item Any finite simplicial complex can be realized as a simplicial complex \(\mathcal{S}(\rho)\) for some tuple \(\rho\) of positive integers (see Proposition~\ref{proposition:realization}). Any map of finite simplicial complexes which does not contract faces can be realized as a weighted simplicial map \((\mathcal{S}(\rho), \Phi_{\mathcal{S}(\rho)}) \rightarrow (\mathcal{S}(\mu), \Phi_{\mathcal{S}(\mu)})\) for some smooth well-formed weighted complete intersection \(X_{\mu} \subset \mathbb{P}(\rho)\) (see Example~\ref{example:realization}).
  \end{enumerate}
\end{example}

\begin{structure}
  This paper is organized as follows. In Section~\ref{section:preliminaries} we briefly remind basic facts about weighted complete intersections. In Section~\ref{section:WSC} we describe our generalization of the approach in~\cite{przyjalkowski/nef}. In Section~\ref{section:proof} we prove Theorem~\ref{theorem:main} and auxiliary statements used in construction of Example~\ref{example:ill-defined}.
\end{structure}

\begin{acknowledgements}
  This work was supported by the Russian Science Foundation under grant no.~19--11--00164, \url{https://rscf.ru/en/project/19-11-00164/}.

  The author is grateful to V.~Przyjalkowski and C.~Shramov for helpful suggestions. We also want to thank the referee for their useful remarks.
\end{acknowledgements}

\section{Preliminaries}\label{section:preliminaries}

Let us remind the basics on weighted projective spaces.

\begin{notation}\label{notation:polynomial-ring}
  Let  \(\rho = (a_0, \ldots, a_N)\) be a tuple of positive integers. We denote by \(R^{\rho} = \mathbb{C}[X_0, \ldots, X_N]\) the polynomial ring over \(\mathbb{C}\) with the following structure of a graded ring:
  \[
    \deg \left ( \prod_{i = 0}^N X_i^{\alpha_i} \right ) = \sum_{i = 0}^N a_i \alpha_i, \quad
    (\alpha_0, \ldots, \alpha_N) \in \mathbb{Z}_{\geqslant 0}^{N + 1}, \quad
    R^{\rho} = \bigoplus_{n = 0}^{\infty} R^{\rho}_n,
  \]
  where \(R^{\rho}_n\) is the \(n\)-th graded component of the ring \(R^{\rho}\).
\end{notation}

\begin{definition}
  Let  \(\rho = (a_0, \ldots, a_N)\) be a tuple of positive integers. We refer to \(\mathbb{P}(\rho) = \Proj(R^{\rho})\) as the \emph{weighted projective space with weights \(\rho\)}.
\end{definition}

\begin{definition}[{\cite[Definition~5.11]{ianofletcher/weighted}}]
  A weighted projective space \(\mathbb{P}(a_0, \ldots, a_N)\) is said to be \emph{well-formed} if
  \(
  \gcd(a_0, \ldots, a_{i - 1}, \widehat{a_i}, a_{i + 1}, \ldots, a_N) = 1
  \)
  for any \(i = 0, \ldots, N\).
\end{definition}

\begin{proposition}[{\cite[\nopp 1.3.1]{dolgachev/weighted}}]
  Any weighted projective space is isomorphic to a well-formed one.
\end{proposition}

\begin{definition}\label{definition:coordinate-stratum}
  Let \(\mathbb{P}(\rho) = \Proj(\mathbb{C}[X_0, \ldots, X_N])\) be a weighted projective space with weights \(\rho = (a_0, \ldots, a_N)\). For any subset \(I \subset \{0, \ldots, N\}\) the corresponding \emph{coordinate stratum} is the following closed subvariety:
  \[
    \mathbb{P}(\rho)_I = \Proj(R^{\rho} / \langle \{X_i \mid i \not \in I\} \rangle) \subset \mathbb{P}(\rho).
  \]
\end{definition}

\begin{lemma}[{\cite[\nopp 5.15]{ianofletcher/weighted}}]\label{lemma:WPS-WF-singular-locus}
  Let \(\mathbb{P}(\rho)\) be a well-formed weighted projective space with weights \(\rho = (a_0, \ldots, a_N)\). Then \(\Sing(\mathbb{P}(\rho))\) is the union of coordinate strata
  \[
    \Sing(\mathbb{P}(\rho)) = \bigcup_I \mathbb{P}(\rho)_I
  \]
  over all subsets \(I \subset \{0, \ldots, N\}\) such that \(a_I = \gcd(\{a_i \mid i \in I\}) > 1\). Moreover, for any subset \(I \subset \{0, \ldots, N\}\) a general point of the coordinate stratum \(P \in \mathbb{P}(\rho)_I \subset \mathbb{P}(\rho)\) is a cyclic quotient singularity of index \(a_I\).
\end{lemma}

One can explicitly describe the base locus of linear systems \(\vert \mathcal{O}_{\mathbb{P}(\rho)}(d) \vert\).

\begin{lemma}[{\cite[Theorem~1.4.1]{dolgachev/weighted}}]\label{lemma:WPS-WF-On-sheaves}
  Let \(\mathbb{P}(\rho) = \Proj(R^{\rho})\) be a well-formed weighted projective space. Then we can identify \(H^0(\mathbb{P}(\rho), \mathcal{O}_{\mathbb{P}(\rho)}(n)) \simeq R^{\rho}_n\) for any \(n \in \mathbb{Z}_{\geqslant 0}\).
\end{lemma}

\begin{corollary}\label{corollary:WPS-WF-singular-locus-O1-base-locus}
  Let \(\mathbb{P}(\rho)\) be a well-formed weighted projective space. Then we have \(\Sing(\mathbb{P}(\rho)) \subset \Bs(\vert \mathcal{O}_{\mathbb{P}(\rho)}(1) \vert)\).
\end{corollary}

\begin{definition}[see~{\cite{przyjalkowski/weighted}} or~{\cite[Definition~2.12]{ovcharenko/classification}}]\label{definition:WCI}
  We refer to a closed subscheme \(X \subset \mathbb{P}(\rho)\) as a \emph{weighted complete intersection} of multidegree \(\mu = (d_1, \ldots, d_c) \in \mathbb{Z}_{> 0}^c\) if the corresponding saturated ideal \(I_X \subset R^{\rho}\) is generated by a regular sequence \((f_1, \ldots, f_c)\) of weighted homogeneous polynomials of degrees \(\deg(f_j) = d_j\).
\end{definition}

Let us recall basic geometrical properties of weighted complete intersections.

\begin{definition}[{\cite[Definition~6.3]{ianofletcher/weighted}}]\label{definition:WPV-quasismooth}
  A closed subscheme \(X \subset \mathbb{P}(\rho)\) is said to be \emph{quasi-smooth} if its affine cone \(\Spec(R^{\rho} / I_X)\) is smooth outside the origin, where \(I_X \subset R^{\rho}\) is the corresponding saturated ideal.
\end{definition}

\begin{definition}[cf.~{\cite[Definition~1.1]{dimca/singularities}}]\label{definition:WF-subscheme}
  A closed subscheme \(X \subset \mathbb{P}(\rho)\) is said to be \emph{well-formed} if \(\mathbb{P}(\rho)\) is well-formed, and \(\codim_X (X \cap \Sing(\mathbb{P}(\rho))) \geqslant 2\).
\end{definition}

\begin{proposition}[see~{\cite[Proposition~8]{dimca/singularities}} and {\cite[Corollary~2.14]{przyjalkowski/bounds}}]\label{proposition:WCI-smooth-WF-criterion}
  Let \(X \subset \mathbb{P}(\rho)\) be a weighted complete intersection. The following assertions are equivalent:
  \begin{itemize}
  \item \(X\) is smooth and well-formed;
  \item \(\mathbb{P}(\rho)\) is well-formed, \(X\) is quasi-smooth, and \(X \cap \Sing(\mathbb{P}(\rho)) = \varnothing\).
  \end{itemize}
\end{proposition}

\begin{example}\label{example:WCI-non-WF-pathologies}
  If a weighted complete intersection is not well-formed, various pathologies can arise.
  \begin{itemize}
  \item There exists a K3-surface which can be realized as a smooth and quasi-smooth hypersurface \(X\) of degree 9 in \(\mathbb{P}(1, 2, 2, 3)\) which is not well-formed (see~\cite[6.15(ii)]{ianofletcher/weighted}). The naive adjunction formula (see~{\cite[Theorem~3.3.4]{dolgachev/weighted}},~{\cite[\nopp 6.14]{ianofletcher/weighted}}) would imply that \(\omega_X \simeq \mathcal{O}_X(1)\), which is nonsense.
  \item A general weighted hypersurface of multidegree \(6\) in \(\mathbb{P}(2, 3, 5^{(t)})\) is not well-formed or quasi-smooth for any \(t > 0\); nonetheless, it is smooth (see~\cite[Example~2.9]{przyjalkowski/on-automorphisms}).
  \end{itemize}
\end{example}

\begin{definition}[{\cite[Definition~6.5]{ianofletcher/weighted}}]\label{definition:WCI-linear-cone}
  A weighted complete intersection \(X \subset \mathbb{P}(a_0, \ldots, a_N)\) of multidegree \((d_1, \ldots, d_c)\) is \emph{an intersection with a linear cone} if we have \(a_i = d_j\) for some \(i\) and \(j\). 
\end{definition}

\begin{remark}
  Not to be an intersection with a linear cone is a rather mild restriction, provided that a weighted complete intersection is general and quasi-smooth (see~\cite[Proposition~2.9]{przyjalkowski/codimension}).
\end{remark}

Let \(X \subset \mathbb{P}(\rho)\) be a weighted complete intersection of multidegree \(\mu\). We know from Proposition~\ref{proposition:WCI-smooth-WF-criterion} that it is smooth and well-formed if and only if it is quasi-smooth, and \(X \cap \Sing(\mathbb{P}(\rho)) = \varnothing\). There exist necessary and sufficient conditions for quasi-smoothness of a general weighted complete intersection in terms of its weights and degrees (see~\cite[Proposition~3.1]{pizzato/nonvanishing}). Technically speaking, a necessary and sufficient combinatorial condition for \(X \cap \Sing(\mathbb{P}(\rho)) = \varnothing\) can be extracted from Lemma~\ref{lemma:WPS-WF-singular-locus} and~\cite[Lemma~A.12]{ovcharenko/classification}. But such a characterization is unusable, so we prefer to use the following necessary condition.

\begin{definition}[see~{\cite[Definition~2.48]{ovcharenko/classification}}, cf.~{\cite[Definition~4.1]{pizzato/nonvanishing}}]\label{definition:strictly-regular-pair}
  Let \(\rho = (a_0, \ldots, a_N)\) and \(\mu = (d_1, \ldots, d_c)\) be tuples of positive integers. We refer to \((\rho; \mu)\) as a \emph{strictly regular pair} if for any subset \(I \subset \{0, \ldots, N\}\) such that \(\gcd(\{a_i \mid i \in I\}) > 1\) there exists a subset \(J \subset \{1, \ldots, c\}\), where \(\vert J \vert = \vert I \vert\),  with the following property: for any \(j \in J\) there exist numbers \(\{n_{i, j} \in \mathbb{Z}_{\geqslant 0} \mid i \in I\}\) such that \(d_j = \sum_{i \in I} n_{i, j} a_i\). 
\end{definition}

\begin{proposition}[see~{\cite[Corollary~2.49]{ovcharenko/classification}}, cf.~{\cite[Lemma~2.15]{przyjalkowski/bounds}},~{\cite[Proposition~4.1]{chen/quasismooth}}]\label{proposition:WCI-smooth-WF-combinatorics}
  Let \(X \subset \mathbb{P}(\rho)\) be a smooth well-formed weighted complete intersection of multidegree \(\mu\) which is not an intersection with a linear cone. Then the pair \((\rho; \mu)\) is strictly regular.
\end{proposition}

\begin{remark}\label{remark:strictly-regular-pair}
  Let \(\mathbb{P}(\rho)\) be a well-formed weighted projective space. Then a pair \((\rho; \mu)\) is strictly regular if and only if for any coordinate stratum \(\mathbb{P}(\rho)_I \subset \Sing(\mathbb{P}(\rho))\) there exists a subset \(J \subset \{1, \ldots, c\}\) of cardinality \(\vert I \vert\) such that \(\mathbb{P}(\rho)_I \not \subset \Bs(\vert \mathcal{O}_{\mathbb{P}(\rho)}(d_j) \vert)\) for any \(j \in J\).    
\end{remark}

\section{Weighted simplicial complexes}\label{section:WSC}

In this section we investigate to what extent we can generalize the approach of the proof of Theorem~\ref{theorem:nef}. The main observation of~\cite{przyjalkowski/nef} is that the singular locus of a weighted projective space can be interpreted as a \emph{weighted simplicial complex}, i.e., an abstract simplicial complex equipped with a weight function.

\begin{definition}[{cf.~\cite[Definition~3.1]{przyjalkowski/nef}}]\label{definition:WSC}
  Let \(\mathcal{C}\) be a finite abstract simplicial complex (see~\cite[Definition~5.1.1]{bruns/cohenmacaulay}). A \emph{weighted simplicial complex} is a pair \((\mathcal{C}, \Phi_{\mathcal{C}})\), where \(\Phi_{\mathcal{C}} \colon \mathcal{C} \rightarrow \mathbb{Z}_{> 0}\) is a map satisfying \(\Phi_{\mathcal{C}}(I) = \gcd_{i \in I}(\Phi_{\mathcal{C}}(\{i\}))\) for any face \(I \in \mathcal{C}\). We will refer to the map \(\Phi_{\mathcal{C}}\) as the \emph{weight function} of \((\mathcal{C}, \Phi_{\mathcal{C}})\).

  A \emph{weighted simplicial map} \(F \colon (\mathcal{C}, \Phi_{\mathcal{C}}) \rightarrow (\mathcal{D}, \Phi_{\mathcal{D}})\) is a simplicial map \(F \colon \mathcal{C} \rightarrow \mathcal{D}\) such that \(\Phi_{\mathcal{C}}(I) \mid \Phi_{\mathcal{D}}(F(I))\) for any face \(I \in \mathcal{C}\).

  We denote by \(\Vertices(\mathcal{C})\) the set of all vertices of the abstract simplicial complex \(\mathcal{C}\). Then we have \(\mathcal{C} \subset 2^{\Vertices(\mathcal{C})}\), where \(2^{\Vertices(\mathcal{C})}\) is the set of all subsets of the set \(\Vertices(\mathcal{C})\).
\end{definition}

\begin{remark}
  Note that a weight function \(\Phi_{\mathcal{C}}\) is determined by its values on vertices of the face. We do not require that a subset \(I \subset \Vertices(\mathcal{C})\) with \(\gcd_{i \in I}(\Phi_{\mathcal{C}}(\{i\})) > 1\) should automatically be a face.
\end{remark}

\begin{example}
  Let \(S = \{0, 1, 2, 3\}\) be a set of cardinality 4. Following Figure~\ref{figure:example}, we can introduce on the set \(S\) the following structure of a weighted simplicial complex \((\mathcal{C}, \Phi_{\mathcal{C}})\):
  \[
    \mathcal{C} = \{I \subset \{0, 1, 2, 3\} \mid \gcd_{i \in I}(\Phi_i) > 1\}; \quad
    \Phi_{\mathcal{C}}(i) = \Phi_i, \quad i = 0, \ldots, 3.
  \]
\end{example}

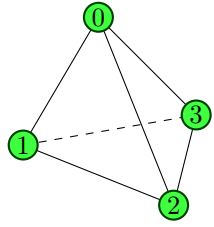
\begin{figure}[H]
  \centering
  \begin{minipage}{.33\textwidth}
    \centering
    \begin{tikzpicture}[vertex/.style={inner sep=1pt,circle,draw=green!25!black,fill=green!75!white,thick}]
      \coordinate (V1) at (4,2.5);
      \coordinate (V2) at (3,.8);
      \coordinate (V3) at (5,0);
      \coordinate (V4) at (5.3,1.2);
  
      \draw (V1) -- (V2) -- (V3) -- (V4) -- cycle;
      \draw (V1) -- (V3);
      \draw[dashed] (V2) -- (V4);
      
      \node[vertex] at (V1){0};
      \node[vertex] at (V2){1};
      \node[vertex] at (V3){2};
      \node[vertex] at (V4){3};
    \end{tikzpicture} 
  \end{minipage}%
  \centering
  \begin{minipage}{.33\textwidth}
    \centering
    \begin{gather*}
      \Phi_0 = b_1 b_2 b_5, \quad
      \Phi_1 = b_1 b_4 b_6, \\
      \Phi_2 = b_3 b_4 b_5, \quad
      \Phi_3 = b_2 b_3 b_6; \\
      b_i \in \mathbb{Z}_{> 1}, \quad i = 1, \ldots, 6; \\
      \gcd(b_j, b_k) = 1, \quad j \neq k.
    \end{gather*}
  \end{minipage}
  \caption{The 1-skeleton of a tetrahedron as a weighted simplicial complex \((\mathcal{C}, \Phi_{\mathcal{C}})\).}
  \label{figure:example}
\end{figure}

\begin{definition}[{\cite[Definition~5.1.2]{bruns/cohenmacaulay}}]
  Let \(\mathcal{C}\) be a finite abstract simplicial complex. Fix a numeration of vertices: \(\Vertices(\mathcal{C}) \simeq \{0, \ldots, N\}\). The \emph{Stanley--Reisner ring} \(\mathbb{C}[\mathcal{C}]\) is the quotient of the polynomial ring \(\mathbb{C}[X_0, \ldots, X_N]\) by the ideal \(I_{\mathcal{C}}\) generated by square-free monomials corresponding to non-faces of \(\mathcal{C}\):
  \[
    I_{\mathcal{C}} = (x_{i_1} \cdots x_{i_r} \mid \{i_1, \ldots, i_r\} \not \in \mathcal{C}).
  \]
\end{definition}

\begin{definition}
  Let \((\mathcal{C}, \Phi_{\mathcal{C}})\) be a weighted simplicial complex, and let us fix a numeration of vertices: \(\Vertices(\mathcal{C}) \simeq \{0, \ldots, N\}\). The \emph{graded Stanley--Reisner ring} \(\mathbb{C}[(\mathcal{C}, \Phi_{\mathcal{C}})]\) is the Stanley--Reisner ring \(\mathbb{C}[\mathcal{C}]\) of the abstract simplicial complex \(\mathcal{C}\) equipped with the grading \(\deg(X_i) = \Phi_{\mathcal{C}}(\{i\})\).
\end{definition}

\begin{remark}
  Let \((\mathcal{C}, \Phi_{\mathcal{C}})\) be a weighted simplicial complex. Fix a numeration of vertices: \(\Vertices(\mathcal{C}) \simeq \{0, \ldots, N\}\). By construction the graded Stanley--Reisner ring \(\mathbb{C}[(\mathcal{C}, \Phi_{\mathcal{C}})]\) is a quotient of the graded polynomial ring \(\mathbb{C}[X_0, \ldots, X_N]\) by the monomial ideal \(I_{\mathcal{C}}\). Then the quotient map \(\mathbb{C}[X_0, \ldots, X_N] \twoheadrightarrow \mathbb{C}[(\mathcal{C}, \Phi_{\mathcal{C}})]\) induces the inclusion \(\Proj(\mathbb{C}[(\mathcal{C}, \Phi_{\mathcal{C}})]) \hookrightarrow \mathbb{P}(a_0, \ldots, a_N)\), where we put \(a_i = \Phi_{\mathcal{C}}(\{i\})\) for any \(i = 0, \ldots, N\).
\end{remark}

Consequently, we can describe various configurations of coordinate strata (in the sense of Definition~\ref{definition:coordinate-stratum}) in terms of weighted simplicial complexes. Our main examples would be the singular locus \(\Sing(\mathbb{P}(\rho))\) of a weighted projective space and base loci of linear systems \(\vert \mathcal{O}_{\mathbb{P}(\rho)}(d) \vert\).

\begin{notation}\label{notation:S-complex}
  Let \(\rho = (a_0, \ldots, a_N)\) be a tuple of positive integers. We denote by \((\mathcal{S}(\rho), \Phi_{\mathcal{S}(\rho)})\) the weighted simplicial complex on the set \(\{0, \ldots, N\}\) defined as follows:
  \[
    \mathcal{S}(\rho) = \{I \subset \{0, \ldots, N\} \mid \gcd_{i \in I}(a_i) > 1\}, \quad
    \Phi_{\mathcal{S}(\rho)}(I) = \gcd_{i \in I}(a_i).
  \]
\end{notation}

\begin{notation}\label{notation:base-locus}
  Let \(\rho = (a_0, \ldots, a_N)\) be a tuple of positive integers. For any \(d \in \mathbb{Z}_{> 0}\) we denote by \(\mathcal{B}_{\rho}(d)\) the abstract simplicial complex on the set \(\{0, \ldots, N\}\), where \(I \subset \{0, \ldots, N\}\) is a face if and only if we \emph{cannot} present the number \(d\) in the form \(d = \sum_{i \in I} n_i a_i\) for some non-negative integers \(\{n_i \in \mathbb{Z}_{\geqslant 0} \mid i \in I\}\). We equip \(\mathcal{B}_{\rho}(d)\) with the weight function \(\Phi_{\mathcal{B}_{\rho}(d)}\) defined by \(\Phi_{\mathcal{B}_{\rho}(d)}(\{i\}) = a_i\) for any \(i = 0, \ldots, N\).
\end{notation}

\begin{lemma}\label{lemma:face-ring-proj}
  Let \(\mathbb{P}(\rho) = \Proj(R^{\rho})\) be a well-formed weighted projective space, where \(\rho = (a_0, \ldots, a_N)\). Then
  \[
    \Sing(\mathbb{P}(\rho)) = \Proj(\mathbb{C}[(\mathcal{S}(\rho), \Phi_{\mathcal{S}(\rho)})]); \quad
    \Bs(\vert \mathcal{O}_{\mathbb{P}(\rho)}(d) \vert)^{\red} = \Proj(\mathbb{C}[(\mathcal{B}_{\rho}(d), \Phi_{\mathcal{B}_{\rho}(d)})]), \quad
    d \in \mathbb{Z}_{> 0}.
  \]
\end{lemma}

\begin{proof}
  The first part is a a reformulation of the first statement of Lemma~\ref{lemma:WPS-WF-singular-locus}. As for the second part, Lemma~\ref{lemma:WPS-WF-On-sheaves} provides the description of \(\Bs(\vert \mathcal{O}_{\mathbb{P}(\rho)}(d) \vert)^{\red}\) as the locus of common zeroes of all weighted monomials of the degree \(d\). Then it is clear that the statement is a reformulation of Notation~\ref{notation:base-locus}.
\end{proof}

\begin{example}
  Lemma~\ref{lemma:face-ring-proj} and Figure~\ref{figure:complexes} describe the singular locus \(\Sing(\mathbb{P}(\rho))\), where \(\rho = (1^{(t + 1)}, 6, 10, 15)\) for any \(t \geq 0\), and base loci of linear systems \(\Bs(\vert \mathcal{O}_{\mathbb{P}(\rho)}(16) \vert)\), \(\Bs(\vert \mathcal{O}_{\mathbb{P}(\rho)}(21) \vert)\), and \(\Bs(\vert \mathcal{O}_{\mathbb{P}(\rho)}(25) \vert)\).
\end{example}

\begin{figure}[H]
  \centering
  \begin{minipage}{.25\textwidth}
    \centering
    \begin{tikzpicture}[scale=0.625000,vertex/.style={inner sep=1pt,circle,draw=green!25!black,fill=green!75!white,thick}]
      \foreach \i in {1, ..., 3}
      \fill (\i*360/3:2) coordinate (V\i) circle(2.4 pt);
		
      \draw (V1)--(V2)--(V3)--(V1);
	
      \node[vertex] at (V1){\(t + 1\)};
      \node[vertex] at (V2){\(t + 2\)};
      \node[vertex] at (V3){\(t + 3\)};
    \end{tikzpicture} 
  \end{minipage}%
  \centering
  \begin{minipage}{.25\textwidth}
    \centering
    \begin{tikzpicture}[scale=0.625000,vertex/.style={inner sep=1pt,circle,draw=green!25!black,fill=green!75!white,thick}]
      \foreach \i in {1, ..., 3}
      \fill (\i*360/3:2) coordinate (V\i) circle(2.4 pt);
		
      \draw (V2)--(V3)--(V1);
	
      \node[vertex] at (V1){\(t + 1\)};
      \node[vertex] at (V2){\(t + 2\)};
      \node[vertex] at (V3){\(t + 3\)};
    \end{tikzpicture} 
  \end{minipage}%
  \centering
  \begin{minipage}{.25\textwidth}
    \centering
    \begin{tikzpicture}[scale=0.62500,vertex/.style={inner sep=1pt,circle,draw=green!25!black,fill=green!75!white,thick}]	
      \foreach \i in {1, ..., 3}
      \fill (\i*360/3:2) coordinate (V\i) circle(2.4 pt);
		
      \draw (V1)--(V2)--(V3);
	
      \node[vertex] at (V1){\(t + 1\)};
      \node[vertex] at (V2){\(t + 2\)};
      \node[vertex] at (V3){\(t + 3\)};
    \end{tikzpicture} 
  \end{minipage}%
  \centering
  \begin{minipage}{.25\textwidth}
    \centering
    \begin{tikzpicture}[scale=0.62500,vertex/.style={inner sep=1pt,circle,draw=green!25!black,fill=green!75!white,thick}]
      \foreach \i in {1, ..., 3}
      \fill (\i*360/3:2) coordinate (V\i) circle(2.4 pt);
		
      \draw (V3)--(V1)--(V2);
	
      \node[vertex] at (V1){\(t + 1\)};
      \node[vertex] at (V2){\(t + 2\)};
      \node[vertex] at (V3){\(t + 3\)};
    \end{tikzpicture} 
  \end{minipage}
  \caption{Weighted simplicial complexes \(\mathcal{S}(\rho)\), \(\mathcal{B}_{\rho}(16)\), \(\mathcal{B}_{\rho}(21)\), \(\mathcal{B}_{\rho}(25)\), where \(\rho = (1^{(t + 1)}, 6, 10, 15)\).}
  \label{figure:complexes}
\end{figure}

\begin{definition}\label{definition:face-poset}
  Let \(\mathcal{C}\) be an abstract simplicial complex. We denote by \((\mathcal{C}, \supseteq)\) its \emph{face poset}, i.e., the set of all faces of the complex \(\mathcal{C}\) partially ordered with respect to the inclusion.
\end{definition}

\begin{proposition}[see Section~\ref{section:proof} for the proof]\label{proposition:poset-map}
  Let \(X \subset \mathbb{P}(\rho)\) be a smooth well-formed weighted complete intersection of multidegree \(\mu = (d_1, \ldots, d_c)\) which is not an intersection with a linear cone, where we put \(\rho = (a_0, \ldots, a_N)\). There exists an order-preserving map of face posets \(\chi \colon (\mathcal{S}(\rho), \supseteq) \rightarrow (\mathcal{S}(\mu), \supseteq)\) such that
  \begin{enumerate}
  \item We have \(\vert \chi(I) \vert = \vert I \vert\) for any face \(I \in \mathcal{S}(\rho)\).
  \item For any face \(I \in \mathcal{S}(\rho)\) and any vertex \(j \in \chi(I)\) we have \(\mathbb{P}(\rho)_I \not \subset \Bs(\vert \mathcal{O}_{\mathbb{P}(\rho)}(d_j) \vert)\). In particular, for any \(j = \Vertices(\mathcal{S}(\mu)) = \{1, \ldots, c\}\) and \(I \in \mathcal{S}(\rho)\) such that \(I \subset \chi^{-1}(j)\) we have \(\mathbb{P}(\rho)_I \cap \Bs(\vert \mathcal{O}_{\mathbb{P}(\rho)}(d_j) \vert)= \varnothing\).
  \item The associated vertex map does not contract any edge \(I = \{i_1, i_2\} \in \mathcal{S}(\rho)\) with \(a_{i_1} \mid a_{i_2}\).
  \end{enumerate}
\end{proposition}

\begin{remark}
  Formally speaking, in Notation~\ref{notation:S-complex} we have started the numeration from \(0\), and the numeration of the tuple \(\mu = (d_1, \ldots, d_c)\) starts from \(1\). We hope that it would not confuse the reader.
\end{remark}

\begin{remark}
  The notion of a face poset is a tautology, one can \emph{define} an abstract simplicial complex as a kind of poset (cf.~\cite[Definition~5.1.1]{bruns/cohenmacaulay}). The difference here is that we are interested in ``generalized simplicial maps'' in the following sense. We require only that the image of a subface should be contained in the image of a face, yet the image of a face \emph{is not defined} by the images of its vertices.
\end{remark}

\begin{example}
  The result of applying of Proposition~\ref{proposition:poset-map} to Example~\ref{example:no-strong-nef-partition} is depicted at Figure~\ref{figure:poset-map}.
\end{example}

\begin{figure}[H]
  \centering
  \begin{minipage}{.33\textwidth}
    \centering
    \begin{tikzpicture}[scale=0.625000,vertex/.style={inner sep=1pt,circle,draw=green!25!black,fill=green!75!white,thick}]
      \foreach \i in {1, ..., 3}
      \fill (\i*360/3:2) coordinate (V\i) circle(2.4 pt);
		
      \draw (V1)--(V2)--(V3)--(V1);
	
      \node[vertex] at (V1){\(t + 1\)};
      \node[vertex] at (V2){\(t + 2\)};
      \node[vertex] at (V3){\(t + 3\)};
    \end{tikzpicture} 
  \end{minipage}%
  \centering
  \begin{minipage}{.33\textwidth}
    \centering
    \begin{tikzpicture}[scale=0.625000,vertex/.style={inner sep=1pt,circle,draw=cyan!25!black,fill=cyan!75!white,thick}]
      \foreach \i in {1, ..., 3}
      \fill (\i*360/3:2) coordinate (V\i) circle(2.4 pt);
      \fill (0:0) coordinate (V0) circle(2.4 pt);

      \draw (V0)--(V1);
      \draw (V0)--(V2);
      \draw (V0)--(V3);
	
      \node[vertex] at (V0){\(4\)};
      \node[vertex] at (V1){\(1\)};
      \node[vertex] at (V2){\(2\)};
      \node[vertex] at (V3){\(3\)};
    \end{tikzpicture} 
  \end{minipage}%
  
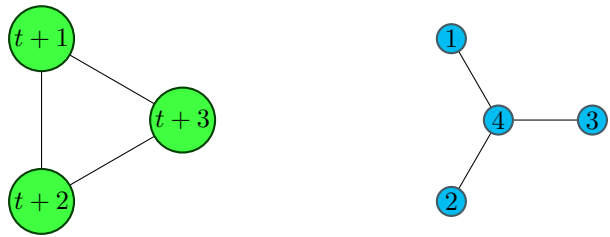
\centering
  \begin{gather*}
    \{t + 1\}, \{t + 2\}, \{t + 3\} \mapsto \{4\}, \; \{t + 1, t + 2\} \mapsto \{1, 4\}, \\
    \{t + 1, t + 3\} \mapsto \{2, 4\}, \; \{t + 2, t + 3\} \mapsto \{3, 4\}; \\
    \Phi_{\mathcal{S}(\rho)}(\{t + 1\}) = 6, \; \Phi_{\mathcal{S}(\rho)}(\{t + 2\}) = 10, \;
    \Phi_{\mathcal{S}(\rho)}(\{t + 3\}) = 15; \\
    \Phi_{\mathcal{S}(\mu)}(\{1\}) = 16, \; \Phi_{\mathcal{S}(\mu)}(\{2\}) = 21, \;
    \Phi_{\mathcal{S}(\mu)}(\{3\}) = 25, \; \Phi_{\mathcal{S}(\mu)}(\{4\}) = 30.
  \end{gather*}
  \caption{An order-preserving map of face posets \(\chi \colon (\mathcal{S}(\rho), \supseteq) \rightarrow (\mathcal{S}(\mu), \supseteq)\), where we put \(\rho = (1^{(t + 1)}, 6, 10, 15)\) and \(\mu = (16, 21, 25, 30)\) for any \(t \geqslant 0\).}
  \label{figure:poset-map}
\end{figure}

The following examples show that the combinatorics of smooth well-formed weighted complete intersections can be arbitrarily complicated from the point of view of simplicial geometry.

\begin{proposition}[see Section~\ref{section:proof} for the proof]\label{proposition:realization}
  Let \(\mathcal{C}\) be a finite abstract simplicial complex. Then there exists a tuple \(\rho\) of positive integers such that \(\mathcal{C}\) is isomorphic to \(\mathcal{S}(\rho)\) as an abstract simplicial complex.
\end{proposition}

\begin{example}\label{example:contraction}
  Let \(\Delta_l(N)\) be the \(l\)-skeleton of the \(N\)-simplex. By Proposition~\ref{proposition:realization} there exists a tuple \(\rho = (a_0, \ldots, a_N)\) such that \(\Delta_l(N) \simeq \mathcal{S}(\rho)\). Fix a numeration of facets: \(\Facets(\mathcal{S}(\rho)) \simeq \{I_1, \ldots, I_w\}\). Put
  \[
    \mu = (2^{(m)}, d_1, \ldots, d_w; \lcm(\rho)^{(l)}), \quad
    d_j = \sum_{i \in I_j} a_i, \quad j = 1, \ldots, w; \quad m \in \mathbb{Z}_{\geqslant 0}.
  \]
  A general weighted complete intersection \(X \subset \mathbb{P}(1^{(t)}, \rho)\) of multidegree \(\mu\) exists for any \(t > m + w + l\) by~\cite[Lemma~2.18]{ovcharenko/classification}. It is smooth and well-formed by Proposition~\ref{proposition:WCI-smooth-WF-criterion}, Lemma~\ref{lemma:WPS-WF-singular-locus}, and~\cite[Proposition~3.1]{pizzato/nonvanishing}. If we choose the tuple \(\rho\) such that \(a_i \nmid d_j\) for any \(i = 0, \ldots, N\) and \(j = 1, \ldots, w\), then the image of any weighted simplicial map \(\mathcal{S}(\rho) \rightarrow \mathcal{S}(\mu)\) is contained in the \((l - 1)\)-simplex \(\{m + w + 1, \ldots, m + w + l\} \in \mathcal{S}(\rho)\).
\end{example}

\begin{proposition}[see Section~\ref{section:proof} for the proof]\label{proposition:inverse}
  Let \(X \subset \mathbb{P}(\rho)\) be a general quasi-smooth weighted complete intersection of multidegree \(\mu\). Assume that there exists a weighted simplicial map \(F \colon \mathcal{S}(\rho) \rightarrow \mathcal{S}(\mu)\) (see Definition~\ref{definition:WSC}) that does not contract any face of the complex \(\mathcal{S}(\rho)\). Then \(X\) is smooth and well-formed.
\end{proposition}

\begin{example}\label{example:realization}
  Let \(F \colon \mathcal{C} \twoheadrightarrow \mathcal{D}\) be any surjective map of finite abstract simplicial complexes which does not contract faces. We claim that \(F\) can be realized as a weighted simplicial map associated with a smooth well-formed weighted complete intersection. By Proposition~\ref{proposition:realization}, we can identify \(\mathcal{C}\) with a weighted simplicial complex \((\mathcal{S}(\rho), \Phi_{\mathcal{S}(\rho)})\) for some tuple \(\rho\) of positive integers. We can endow \(\mathcal{D}\) with a structure of a weighted simplicial complex \((\mathcal{S}(\mu), \Phi_{\mathcal{S}(\mu)})\) as follows: for any vertex we take the least common multiple of all weights lying in the fiber. If the fiber consists of only one weight, we take its arbitrary multiplicity to avoid an intersection with a linear cone (see Definition~\ref{definition:WCI-linear-cone}). A general weighted complete intersection \(X \subset \mathbb{P}(1^{(t)}; \rho)\) of multidegree \((\mu; \lcm(\rho)^{(s)})\) exists for any \(t > m + s\) by~\cite[Lemma~2.18]{ovcharenko/classification} and is quasi-smooth by~\cite[Proposition~3.1]{pizzato/nonvanishing} for \(s \gg 0\). Then it is smooth and well-formed by Proposition~\ref{proposition:inverse}.  
\end{example}

\section{Proof of main results}\label{section:proof}

In this section we follow the notations and definitions of Section~\ref{section:WSC}. We also use \emph{Noetherian induction}.

\begin{principle}[Noetherian induction, cf.~{\cite[Exercise~II.3.16]{hartshorne/geometry}}]
  Let \((X, \preccurlyeq)\) be a well-founded partially ordered set, i.e., such that every non-empty subset \(S \subseteq X\) has a minimal element with respect to \(\preccurlyeq\).

  Let \(P(x)\) be a property of elements \(x \in X\). To prove \(P(x)\) for all elements \(x \in X\), it suffices to prove
  \begin{itemize}\setlength{\itemindent}{2cm}
  \item [\textbf{Induction basis.}] \(P(x)\) is true for all minimal elements \(x \in X\).
  \item [\textbf{Induction step.}] For any non-minimal \(y \in X\) assume that \(P(x)\) is true for all \(x \prec y\). Then \(P(y)\) is true.  
  \end{itemize}
\end{principle}

\begin{proof}[Proof of Proposition~\ref{proposition:poset-map}]
  Put \(\rho = (a_0, \ldots, a_N)\) and \(\mu = (d_1, \ldots, d_c)\). We introduce the following notation:
  \[
    \Phi = \Phi_{\mathcal{S}(\rho)}, \quad
    \mathcal{S(\rho)}^{(b)} = \{i = 0, \ldots, N, \; b \mid a_i\}
    \text{ for any } b \in \im(\Phi).
  \]
  Note that we can reformulate~(2) of Proposition~\ref{proposition:poset-map} in these terms (see Remark~\ref{remark:strictly-regular-pair}). By Proposition~\ref{proposition:WCI-smooth-WF-combinatorics} for any element \(b \in \im(\Phi)\) there exists an injective simplicial map
  \[
    \chi^{(b)} \colon \mathcal{S}(\rho)^{(b)} \hookrightarrow
    \Vertices(\mathcal{S}(\mu)) = \{1, \ldots, c\}
  \]
  such that the face \(\mathcal{S}(\rho)^{(b)} \in \mathcal{S}(\rho)\) does not belong to the simplicial complex \(\mathcal{B}_{\rho}(d_j)\) for any vertex \(j \in \im(\chi^{(b)})\). We call the map \(\chi^{(b)} \colon \mathcal{S}(\rho)^{(b)} \hookrightarrow
    \Vertices(\mathcal{S}(\mu))\) an \emph{admissible injection} (cf.~(1) of Proposition~\ref{proposition:poset-map}).

  In order to prove the statement, we have to correspond to any \(b \in \im(\Phi)\) an admissible injection \(\chi^{(b)}\) in such a way that for any elements \(q_1, q_2 \in \im(\Phi)\) with \(q_1 \mid q_2\) the admissible injection \(\chi^{(q_2)}\) can be obtained as a restriction of \(\chi^{(q_1)}\). We are going to ensure this by Noetherian induction on elements of \(\im(\Phi)\) partially ordered by divisibility. Note that the fibers of the associated simplicial map \(\chi_{\vertices} \colon \mathcal{S}(\rho) \rightarrow \mathcal{S}(\mu)\) then would satisfy the following property: for any \(i_1, i_2 \in \chi_{\vertices}^{-1}(j)\) we have \(a_{i_1} \nmid a_{i_2}\) (cf.~(3) of Proposition~\ref{proposition:poset-map}).
   
  \textbf{Induction basis.} Let \(b \in \im(\Phi)\) be an element such that there exists no element \(q \in \im(\Phi) \setminus \{b\}\) with \(q \mid b\). By assumption there exists an admissible injection \(\chi^{(b)} \colon \mathcal{S}(\rho)^{(b)} \hookrightarrow \Vertices(\mathcal{S}(\mu))\), and we are done.

  \textbf{Induction step.} Let \(b \in \im(\Phi)\) be an element, and assume that we have constructed the required admissible injections \(\chi^{(q)} \colon \mathcal{S}(\rho)^{(q)} \hookrightarrow \Vertices(\mathcal{S}(\mu))\) for any element \(q \in \im(\Phi) \setminus \{b\}\) such that \(q \mid b\), and all of them agree in the sense above. Let \(\{b_1, \ldots, b_l\}\) be the set of all elements \(b_j \in \im(\Phi) \setminus \{b\}\) such that \(b_j \mid b\) such that there exists no element \(q \in \im(\Phi) \setminus \{b_j, b\}\) with \(b_j \mid q \mid b\). Let \(\chi^{(b_j)} \colon \mathcal{S}(\rho)^{(b_j)} \hookrightarrow \Vertices(\mathcal{S}(\mu))\) and \(\chi^{(b)} \colon \mathcal{S}(\rho)^{(b)} \hookrightarrow \Vertices(\mathcal{S}(\mu))\) be admissible injections. If necessary, we can always modify \(\chi^{(b_j)}\) to ensure that \(\im(\chi^{(b)}) \subset \im(\chi^{(b_j)})\) for any \(j = 1, \ldots, l\), since \(b_j \mid b\). Note that by induction assumption for any \(q \in \im(\Phi)\) with \(q \mid b_j\) the admissible injection \(\chi^{(b_j)}\) is the restriction of the admissible injection \(\chi^{(q)}\). Then a modification of \(\chi^{(b_j)}\) automatically yields the correct modification of any such \(\chi^{(q)}\) as well, and we are done.
\end{proof}

\begin{proof}[Proof of Theorem~\ref{theorem:main}]
  Let \(X \subset \mathbb{P}(\rho)\) be a smooth well-formed Fano weighted complete intersection of multidegree \(\mu = (d_1, \ldots, d_c)\), where we put \(\rho = (a_0, \ldots, a_N)\). From Proposition~\ref{proposition:poset-map} we obtain the weighted simplicial map \(\chi_{\vertices} \colon \mathcal{S}(\rho) \rightarrow \mathcal{S}(\mu)\) such that for any \(j \in \Vertices(\mathcal{S}(\mu)) = \{1, \ldots, c\}\) the fiber \(\chi_{\vertices}^{-1}(j)\) is a non-divisible subset of \(\{0, \ldots, N\}\) in the sense of Definition~\ref{definition:divisible-subset}. By assumption any non-divisible subset \(I \subset \{0, \ldots, N\}\) is strongly non-divisible. Moreover, \(\chi_{\vertices} \colon \mathcal{S}(\rho) \rightarrow \mathcal{S}(\mu)\) is a weighted simplicial map, hence \(a_i \mid d_j\) for any \(i \in \chi_{\vertices}^{-1}(j)\). Put \(\Delta_j = d_j - \sum_{i \in \chi_{\vertices}^{-1}(j)} a_i\). Then~\cite[Lemma~6.1]{pizzato/nonvanishing} implies that
  \[
    \Delta_j \geqslant \lcm_{i \in \chi_{\vertices}^{-1}(j)} a_i -
    \sum_{i \in \chi_{\vertices}^{-1}(j)} a_i \geqslant 0, \quad j = 1, \ldots, c.
  \]

  Let \(i_X = \deg(-K_X) > 0\) be Fano index of \(X\). Put \(U = \{i = 0, \ldots, N \mid a_i = 1\}\). The adjunction formula (see~{\cite[Theorem~3.3.4]{dolgachev/weighted}},~{\cite[\nopp 6.14]{ianofletcher/weighted}}) implies that \(\vert U \vert = i_X + \sum_{j = 1}^c \Delta_j\). Consider an arbitrary splitting \(U = U_0 \sqcup \ldots \sqcup U_c\) such that \(\vert U_0 \vert = i_X\), and \(\vert U_j \vert = \Delta_j\) for any \(j = 1, \ldots, c\). Put \(I_j = U_j \sqcup \chi_{\vertices}^{-1}(j)\) for any \(j = 1, \ldots, c\). Then \(I = U_0 \sqcup I_1 \sqcup \ldots \sqcup I_c\) is a required strong nef-partition for \(X\).
\end{proof}

\begin{proof}[Proof of Proposition~\ref{proposition:realization}]
  We are going to endow \(\mathcal{C}\) with a weight function \(\Phi_{\mathcal{C}}\) by Noetherian induction on faces of the simplicial complex \(\mathcal{C}\) partially ordered by reverse inclusion, starting from its facets. Then the tuple \(\rho\) can be obtained as values of the weight function \(\Phi_{\mathcal{C}}\) on vertices of the simplicial complex.

  \textbf{Induction basis.} Let \(\{F_1, \ldots, F_l\} \subset \mathcal{C}\) be the set of all facets of the simplicial complex \(\mathcal{C}\), and \((b_1, \ldots, b_l)\) be a tuple of pairwise coprime integers such that \(b_j > 1\). We put \(\Phi_{\mathcal{C}}(F_j) = b_j\) for any \(j = 1, \ldots, l\).
  
  \textbf{Induction step.} Let \(I \in \mathcal{C}\) be a face, and let \(\{I_1, \ldots, I_k\} \subset \mathcal{C}\) be the set of all faces of the complex \(\mathcal{C}\) such that for any \(j = 1, \ldots, k\) the face \(I_j\) contains \(I\) as its facet. Then we put \(\Phi_{\mathcal{C}}(I) = \lcm_{j = 1, \ldots, k}(\Phi_{\mathcal{C}}(I_j))\).
\end{proof}

\begin{proof}[Proof of Proposition~\ref{proposition:inverse}]
  Put \(\rho = (a_0, \ldots, a_N)\) and \(\mu = (d_1, \ldots, d_c)\). We know from Proposition~\ref{proposition:WCI-smooth-WF-criterion} that a weighted complete intersection \(X\) in a well-formed weighted projective space \(\mathbb{P}(\rho)\) is smooth and well-formed if and only it is quasi-smooth, and \(X \cap \Sing(\mathbb{P}(\rho)) = \varnothing\). Then Lemma~\ref{lemma:WPS-WF-singular-locus} implies that it is enough to check that we have \(X \cap \mathbb{P}(\rho)_I = \varnothing\) for any face \(I \in \mathcal{S}(\rho)\). Recall that by Definition~\ref{definition:coordinate-stratum} we have
  \[
    \mathbb{P}(\rho)_I = \Proj(R^{\rho} / \langle \{X_i \mid i \not \in I\} \rangle), \quad
    R^{\rho} = \mathbb{C}[X_0, \ldots, X_N] = \bigoplus_{n = 0}^{\infty} R^{\rho}_{d_j}, \quad
    \deg(X_i) = a_i,
  \]
  where we are using Notation~\ref{notation:polynomial-ring}. The weighted simplicial map \(F \colon \mathcal{S}(\rho) \rightarrow \mathcal{S}(\mu)\) sends a vertex \(i \in \Vertices(\mathcal{S}(\rho))\) to a vertex \(F(i) \in \Vertices(\mathcal{S}(\mu))\) such that \(a_i \mid d_{F(\{i\})}\). Then there exists a monomial \(X_i^{m_i} \in R^{\rho}\) of degree \(d_{F(\{i\})} = m_i a_i\). By assumption the map \(F\) does not contract any face \(I \in \mathcal{S}(\rho)\), so we have an injection
  \[
    F_I = \left. F \right|_I \colon I \hookrightarrow \Vertices(\mathcal{S}(\mu)) =
    \{1, \ldots, c\}, \quad
    a_i \mid d_{F_I(\{i\})} \text{ for all } i \in I. 
  \]
  In other words, there exists a regular sequence \((X_i^{m_i} \mid i \in I)\) of length \(\vert \im(F_I) \vert = \vert I \vert\), where \(m_i a_i = d_{F_I(\{i\})}\). Then by semi-continuity of codimension (see~\cite[Corollary~A.21]{ovcharenko/classification}) for a general choice of a sequence of homogeneous polynomials \(\overline{f} \in \prod_{j = 1}^c R^{\rho}_{d_j}\) the corresponding ideal \(I_{\overline{f}} \subset R^{\rho} / \langle \{X_i \mid i \not \in I\} \rangle\) has height \(\vert I \vert\). Consequently, for a general choice of \(X\) we have \(X \cap \mathbb{P}(\rho)_I = \varnothing\), so we are done.  
\end{proof}

\printbibliography

\end{document}